\documentclass[12pt,twoside,final]{amsart}
\usepackage[draft=false]{hyperref}
\usepackage[psamsfonts]{amssymb}
\usepackage{times,a4wide}

\usepackage[alphabetic,backrefs]{amsrefs}

\theoremstyle{plain}
\newtheorem*{theorem*}{Theorem}

\newtheorem{theorem}{Theorem}[section]
\newtheorem{proposition}[theorem]{Proposition}
\newtheorem*{maintheorem*}{Main Theorem}
\newtheorem*{proposition*}{Proposition}

\newtheorem*{corollary*}{Corollary}
\newtheorem{lemma}[theorem]{Lemma}
\newtheorem*{lemma*}{Lemma}

\theoremstyle{definition}

\newtheorem*{remark*}{Remark}
\newtheorem*{remarks*}{Remarks}
\newtheorem*{conjecture*}{Conjecture}
\theoremstyle{definition}
\newtheorem{definition}[theorem]{Definition} %[section]

\allowdisplaybreaks

\usepackage[active]{srcltx}
\usepackage{tikz}
\usepackage{pgfplots}
\usepackage{float}
\RequirePackage{pgfplots}
\usetikzlibrary{shadows}
\pgfplotsset{compat=1.13} % XAVIER : 1.16
\newcommand{\D}{\mathbb{D}}
\newcommand{\C}{\mathbb{C}}

\newcommand{\N}{\mathbb{N}}
\newcommand{\R}{\mathbb{R}}
\newcommand{\E}{\mathbb{E}}
\newcommand{\T}{\mathbb{T}}
\renewcommand{\P}{\mathbb{P}}

\newcommand{\Var}{\operatorname{Var}}

\newcommand{\beqa}{\begin{eqnarray*}}
\newcommand{\eeqa}{\end{eqnarray*}}
\newcommand{\bqa}{\begin{eqnarray*}}
\newcommand{\eqa}{\end{eqnarray*}}

\title[Poisson processes in the disk] % and interpolation]
{Inhomogeneous Poisson processes in the disk \\
and interpolation} % and interpolation}

\author[A. Hartmann, X. Massaneda]{Andreas Hartmann,
Xavier Massaneda}

\address{Univ. Bordeaux, CNRS, Bordeaux INP, IMB, UMR 5251, F-33400, Talence, France}

\address{Departament de Matem\`atiques i Inform\`atica,
Universitat  de Barcelona, Gran Via 585, 08007-Bar\-ce\-lo\-na, Catalonia}

\thanks{The first author partially supported by the project REPKA (ANR-18-CE40-0035). Second author partially supported by the Generalitat de Catalunya (grant 2017 SGR 359) and the spanish Ministerio de Ciencia e Innovaci\'on (project PID2021-123405NB-I00).}

\date{\today}


\begin{document}

\begin{abstract} 
We investigate different geometrical properties of the inhomogeneous Poisson point process $\Lambda_\mu$ associated to a positive, locally finite, $\sigma$-finite measure $\mu$ on the unit disk. In particular, we characterize
the processes $\Lambda_\mu$ such that almost surely: 1) $\Lambda_\mu$ is a Carleson-Newman sequence; 2) $\Lambda_\mu$ is the union of a given number $M$ of separated sequences. We use these results to discuss the measures $\mu$ such that the associated process $\Lambda_\mu$ is almost surely an interpolating sequence for the Hardy, Bloch or weighted Dirichlet spaces.
\end{abstract}

\maketitle

\section{Introduction and main results}
Important notions in spaces of analytic functions include zero-sets,  Carleson measures, interpolation, sampling, frames, etc.  Such properties have been studied for many well-known spaces of analytic functions in a deterministic setting. A canonical example is the Hardy space, where all these properties are well established, see \cite{Ga}. 
In other spaces such properties admit theoretical characterizations which are not checkable in general (e.g.\ interpolation
in Dirichlet spaces), see e.g. \cite{S} for a general reference. 
There also exist situations where a general characterization is not
available. In these circumstances it is useful to consider a random setting, which allows to see
whether certain properties are ``generic'' in a sense. The random model we are interested in here is the Poisson point process.

A \emph{Poisson point process} in the unit disk $\D$ is a random sequence $\Lambda$
defined in the following way:
%a kind of random sequence $\Lambda\subset\D$ with the following properties: 
for any Borel set $A\subset \D$  the counting random variable 
$N_A=\# (A\cap\Lambda) $
is well defined and
\begin{itemize}
 \item [(a)] $N_A$ is a Poisson random variable, i.e., there exists $\mu(A)\geq 0$ such that the probability distribution of $N_A$ is
\[
 \P(N_A=k)=e^{-\mu(A)} \frac{(\mu(A))^k}{k!}\ , k\geq 0.
\]
In particular $\E[N_A]=\Var[N_A]=\mu(A)$.

\item [(b)] If $A,B\subset\D$ are disjoint Borel sets then the variables $N_A$, $N_B$ are independent.
 \end{itemize}
It turns out that these two properties uniquely characterize the point process. Also, the values $\mu(A)$ define a $\sigma$-finite Borel measure on $\D$, which is called the \emph{intensity} of the process.

The Poisson process is a well-known statistical model for point distributions with no (or weak) interactions, and it has multiple applications in a great variety of fields \cite{Wi}. Because of property (b), it is clearly
not adequate to describe distributions in which each point is not statistically independent of the other points of the process. For such situations other models have been proposed (e.g. determinantal processes or zeros of Gaussian analytic functions for random sequences with repulsion, or
Cox processes for situations with positive correlations and clumping \cite{HKPV}). 

It is also possible to create a Poisson process from a given, $\sigma$-finite, locally finite, positive Borel measure $\mu$ in $\D$, in the sense that there exists a point process $\Lambda_\mu$ with intensity $\mu$, i.e, whose counting functions satisfy properties (a) and (b) above. This is a well-known, non-trivial fact that can be found, for example, in \cite{La-Pe}*{Theorem 3.6}. Such a Poisson process $\Lambda_\mu$ is sometimes called \emph{inhomogeneous}, or non-stationary.
%\tcr{Je ne suis pas compl\`etement s\^ur d'\'ecrire ceci ici, on a d\'ej

In this paper, given a positive Borel measure $\mu$ on $\D$, we study  elementary geometric properties of the inhomogeneous Poisson process of intensity $\mu$, specifically in relation to conditions used to describe interpolating sequences for various spaces of analytic functions in $\D$. We shall always assume that $\mu(\D)=+\infty$, since otherwise $\Lambda_\mu$ would be finite almost surely.

The probabilistic point of view
%makes even sense when the deterministic setting is well understood, like for instance in the Hardy space, and 
has already been explored before in connection with interpolation. Here we mention Cochran \cite{Coc} and Rudowicz \cite{Ru} who considered the probabilistic model $\Lambda=\{r_n e^{i\theta_n}\}_n$ in which the radii $r_n\subset(0,1)$ are fixed a priori and the arguments $\theta_n$ are chosen uniformly and independently in $[0,2\pi]$ (a so-called Steinhaus sequence). For this model they established a zero-one condition on $\{r_n\}_n$ so that the resulting random sequence is almost surely interpolating for the Hardy spaces. 
In \cite{CHKW} similar results, for the same probabilistic model, were proven for the scale of weighted Dirichlet spaces between the Hardy space and the classical Dirichlet space. See also \cite{DWW} for related 
results in the unit ball and the polydisk.

We express our results in terms of a dyadic discretization of $\mu$.
Consider first the dyadic annuli
\[
 A_n=\{z\in \D:2^{-(n+1)}< 1-|z|\leq 2^{-n}\}, \quad n\geq 0.
\]
Each $A_n$ can be split into $2^n$ boxes of the same size $2^{-n}$: 
\[
 T_{n,k}=\bigl\{z=re^{it}\in A_n: \frac{k}{2^n}\le \frac t{2\pi}<\frac{k+1}{2^n}\bigr\},\quad k=0,1,\ldots,2^n-1.
\]
These boxes can be viewed as the top halves of the Carleson windows
\[
Q(I_{n,k})=\bigl\{z=re^{i\theta}\in \D : r>1-2^{-n}, \, e^{i\theta}\in I_{n,k}\bigr\}
\]
associated to the dyadic intervals
\begin{equation}\label{dy-int}
I_{n,k}=\bigl\{ e^{it}\in\T : \frac{k}{2^n}\le \frac t{2\pi}<\frac{k+1}{2^n}\bigr\}\ ,\quad  n\geq 0\ ,\, k=0,1,\ldots, 2^{n}-1.
\end{equation}

\begin{figure}[H]
\centering
\begin{tikzpicture}[scale=0.6]
\centering
%\shade[top color=blue!15,bottom color=blue!20] (0,0) parabola bend (0,0) (2.2,2.42) parabola bend (1,2.7) (0,2)--(0,0);
\draw [ultra thick] (-4,0) -- (4,0);
\path [thin, blue,fill=blue!7] (-4,4)--(4,4)--(4,8)--(-4,8)--(-4,4);
\draw [thin] (-5,0) -- (-4,0);
\draw [thin] (4,0) -- (5,0);
\draw [thin] (-4,8) -- (4,8);
\draw [thin] (-4,4) -- (4,4);
\draw [thin] (-4,2) -- (4,2);
\draw [thin, dashed] (-4,1) -- (4,1);
\draw [thin, dashed] (-3,0) -- (-3,1);
\draw [thin, dashed] (-1,0) -- (-1,1);
\draw [thin, dashed] (3,0) -- (3,1);
\draw [thin, dashed] (1,0) -- (1,1);
\draw [thin] (-4,0) -- (-4,8);
\draw [thin] (4,0) -- (4,8);
\draw [thin] (0,0) -- (0,4);
\draw [thin] (-2,0) -- (-2,2);
\draw [thin] (2,0) -- (2,2);

\node [below] at (0,0) {$I_{n,k}$};
\node [right] at (4.2,4) {$Q(I_{n,k})$};
\node [above, blue] at (0,5.5) {$T_{n,k}$};

\end{tikzpicture}

\caption{Carleson window $Q(I_{n,k})$ associated to the dyadic interval $I_{n,k}$ and its top half $T_{n,k}$.}
\end{figure}
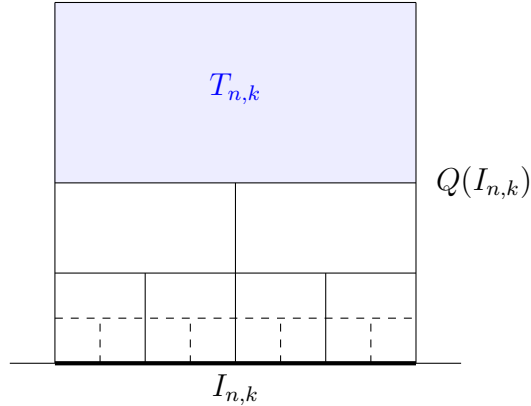

Denote $X_{n,k}=N_{T_{n,k}}$, which by hypothesis is a Poisson random variable of parameter
\[
 \mu_{n,k}:=\E[X_{n,k}]=\Var [X_{n,k}]= \mu(T_{n,k}).
\]
In these terms, the assumption $\mu(\D)=+\infty$ is just
\[
\mu(\D)=\sum_{n\in\N} \sum_{k=0}^{2^n-1} \mu_{n.k}=\sum_{n,k}\mu_{n,k}=+\infty.
\]

A first geometric property on random sequences  we are interested in is separation. For this, we recall that the pseudo-hyperbolic distance in $\D$ is given by
\[
 \rho(z,w)=\left|\frac{z-w}{1-\bar w z}\right|\ \quad z,w\in\D.
\]

\begin{definition}
 A sequence $\Lambda=\{\lambda_k\}_{k\geq 1}\subset \D$ is \emph{separated} if there exists $\delta>0$ such
that 
\[
 \rho(\lambda_k,\lambda_l)\ge\delta,\quad k\neq l.
\]
When we need to specify the separation constant we say that $\Lambda$ is $\delta$-separated.
\end{definition}

We are now in a position to state our result characterizing those $\Lambda_\mu$ which can (almost surely) be expressed as finite unions of separated sequences.

\begin{theorem}\label{thm:separation}
 Let $\Lambda_\mu$ be the Poisson process associated to a positive, %non-atomic, 
$\sigma$-finite, locally finite measure $\mu$ and let $M\geq 1$ be an integer. Then
 \[
  \P\bigl(\Lambda_\mu\ \textrm{union of $M$ separated sequences}\bigr)=
  \begin{cases}
   1\quad \textrm{if}\quad \displaystyle\sum\limits_{n,k}\mu_{n,k}^{M+1}<\infty \\
   0\quad \textrm{if}\quad \displaystyle\sum\limits_{n,k}\mu_{n,k}^{M+1}=\infty.
  \end{cases}
 \]
 In particular,
 \[
  \P\bigl(\Lambda_\mu\ \textrm{separated}\bigr)=
  \begin{cases}
   1\quad \textrm{if}\quad \displaystyle\sum\limits_{n,k}\mu_{n,k}^{2}<\infty \\
   0\quad \textrm{if}\quad \displaystyle\sum\limits_{n,k}\mu_{n,k}^{2}=\infty.
  \end{cases} 
 \]
\end{theorem}

The characterization of a.s. separated sequences was first obtained, with a different proof, in \cite{Ap}*{Teorema 3.2.1}.

Our second result deals with so-called $\alpha$-Carleson sequences. 
Given any arc $I\subset \T=\partial \D$ let $|I|$ denote its normalized length and consider the associated Carleson window
\[
 Q(I)=\bigl\{z=re^{i\theta}\in \D : r>1-|I|, \, e^{i\theta}\in I\bigr\}.
\]

\begin{definition}
 Let $\alpha\in (0, 1]$. The sequence $\Lambda$ satisfies the $\alpha$-\emph{Carleson} condition if there exists $C>0$ such that for all arcs $I\subset\T$
 \[
  \sum_{\lambda\in Q(I)} (1-|\lambda|)^\alpha \leq C |I|^\alpha.
 \]
Such sequences will also be called $\alpha$-Carleson sequences.
\end{definition}

Observe that to check the $\alpha$-Carleson condition it is enough to test on the dyadic intervals $I_{n,k}$ given in \eqref{dy-int}.

The sequences $\Lambda$ satisfying the  $1$-Carleson condition are by far the most studied, because of their r\^ole in the famous
characterization of the interpolating sequences for the algebra $H^\infty$ of bounded holomorphic functions, given 
by L. Carleson \cite{Ca} (see Section~\ref{int_h}). They are sometimes found in the literature under the name of {\it Carleson-Newman} sequences. 

The $\alpha$-Carleson property above is a special case of a more general condition: a finite, positive Borel measure $\sigma$ on $\D$
is a Carleson-measure of order $\alpha\in (0,1]$ if 
$\sigma(Q(I))\le C|I|^\alpha$ for some $C>0$ and all intervals $I$.  As shown by L. Carleson (see e.g. \cite{Ga}), Carleson measures (of order $1$) are precisely those for which the embedding $H^2\subset L^2(\D,\sigma)$ holds; here $H^2$ is the classical Hardy space (see the definition in Subsection \ref{HardyBergman} below). Carleson measures of order $\alpha<1$ have been used, for example, in providing sufficient conditions for solvability of the $\bar\partial_b$-equation in $L^p$, $L^{p,\infty}$ and in Lipschitz spaces of the boundary of strictly pseudoconvex domains \cite{Am-Bo}.

\begin{theorem}\label{thm:Carleson}
  Let $\Lambda_\mu$ be the Poisson process associated to a positive, %non-atomic, 
$\sigma$-finite, locally finite measure $\mu$. Then
  \begin{itemize}
   \item [(a)]
   \[
   \P\bigl(\Lambda_\mu\ \textrm{is a 1-Carleson sequence}\bigr) =
   \begin{cases}
   1\quad \textrm{if there exists $\gamma>1$ such that}\quad \displaystyle\sum\limits_{n,k}\mu_{n,k}^{\gamma}<\infty \\
   0\quad \textrm{if for all $\gamma>1$}\quad \displaystyle\sum\limits_{n,k}\mu_{n,k}^{\gamma}=\infty.
  \end{cases}
   \]
  \end{itemize}
  
  \item [(b)] Let $\alpha\in (0,1)$. If there exists $1<\gamma<\frac 1{1-\alpha}$ such that $ \sum\limits_{n,k} \mu_{n,k}^{\gamma}<+\infty$, then
  \[
 \P\bigl(\Lambda_\mu\  \textrm{is $\alpha$-Carleson}\bigr)=1 
\]

\item [(c)] 
There exists a positive, 
$\sigma$-finite, locally finite measure $\mu$ such that $ \sum\limits_{n,k} \mu_{n,k}^{1/(1-\alpha)}<+\infty$ and
\[
\P\bigl(\Lambda_\mu\  \textrm{is $\alpha$-Carleson}\bigr)=0.
\]

\item [(d)] For every $\gamma>1$ there exists a positive, 
$\sigma$-finite, locally finite measure $\mu$ such that $ \sum\limits_{n,k} \mu_{n,k}^{\gamma}=+\infty$ but 
\[
\P\bigl(\Lambda_\mu\  \textrm{is $\alpha$-Carleson}\bigr)=1
\]
for all $\alpha\in (0,1)$.
\end{theorem}

%An example shows that the power $\gamma<1/(1-\alpha)$ in (b) is optimal, in the sense that there exist mesures $\mu$ for which
%$\sum_{n,k} \mu_{n,k}^{1/(1-\alpha)}<+\infty$ but $\P\bigl(\Lambda_\mu\  \textrm{is $\alpha$-Carleson}\bigr)=0$.

\begin{remarks*} 1) The first statement in part (a) is connected with the first part of the statement in Theorem~\ref{thm:separation}, since it is a  well-known fact that every 1-Carleson (or Carleson-Newman) sequence can be split into a finite number of separated sequences, each of which being of course 1-Carleson\cite{McDS}*{Lemma 21} (obviously a finite number of arbitrary separated sequences may not be Carleson-Newman). However Theorem \ref{thm:Carleson}(a) does not give a precise information on the number of separated sequences involved. It is also mentionable that the condition for a.s.\ separation from Theorem \ref{thm:separation} implies automatically the Carleson condition (picking $\gamma=2>1$). This is perhaps more surprising and may be explained by the nature of the process: the independence of the different points allows for big fluctuations, so the probability of finding pairs of points arbitrarily close is quite big unless the number of points in the process is restricted severely (up to $ \sum_{n,k} \mu_{n,k}^{2}<\infty$). 

2) It is interesting to point out that for the inhomogeneous Poisson process we have a characterization of $1$-Carleson sequences, while in the a priori simpler random model with fixed radii and random arguments there is only a sufficient -- still optimal -- condition (see \cite{CHKW}*{Theorem 1.4}).

3) %According to (a), if the sum diverges for  every $\gamma>1$, then $\Lambda_{\mu}$ is almost surely not $1$-Carleson. This is shown (see the proof below) by proving that for every $N\in\N$ there are infinitely many top-halves $T_{n,k}$ containing at least $N$ points. As a result, if the sum diverges for  every $\gamma>1$, then $\Lambda_{\mu}$ cannot be $\alpha$-Carleson neither.  
%There exist measures $\mu$ for which the sum in (c) converges and  $\P\bigl(\Lambda_\mu\  \textrm{is $\alpha$-Carleson}\bigr)=0$. 
In the case $\alpha\in (0,1)$ the results are less precise than when $\alpha=1$. The value $1/(1-\alpha)$ turns out to be an optimal breakpoint, but nothing specific can be said beyond this value without additional conditions on the distribution of $\mu$. The example given in (c) is part of a certain parameter dependent scale of measures which will be discussed in Section \ref{Examples}, and for which the $\alpha$-Carleson condition is characterized in terms of the parameter. 
%\tcr{On pourrait essayer de voir si on peut obtenir un r\'esultat plus proche d'une caract\'erisation si p.ex. $\mu$ est radiale, mais je n'y crois pas trop.} %Nous n'en avons pas trouv\'e pour Hardy. Pour les r\'esultats sur l'interpolation dans Hardy, Dirichlet, les conditions n\'ecessaires \'etaient une cons\'equence des z\'eros ou de la s\'eparation.}

4) Our conditions, both here and in Theorem~\ref{thm:separation}, are expressed in terms of $\mu_{n,k}$, thus redistributing continuously $\mu$ on $T_{n,k}$ if necessary, we can always assume that $\mu$ is absolutely continuous with respect to the Lebesgue measure.
\end{remarks*}

%The rest of the paper is mainly devoted to examine implications of Theorems~\ref{thm:separation} and ~\ref{thm:Carleson} in the study of random interpolating sequences for various spaces of holomorphic functions. %In particular, we provide a full characterization of the measures $\mu$ that produce almost surely interpolating sequences for the Hardy  or the Bloch space.  

The structure of the paper is as follows. In Section~\ref{theorems} we prove the main Theorems~\ref{thm:separation} and ~\ref{thm:Carleson}. Section~\ref{int_h} deals with the consequences of these results in the study of interpolating sequences for various spaces of holomorphic functions. In particular, we find precise conditions so that a Poisson process $\Lambda_\mu$ is almost surely an interpolating sequence for the Hardy spaces $H^p$, $0<p\leq\infty$, the Bloch space $\mathcal B$, or the Dirichlet spaces $\mathcal D_\alpha$, $\alpha\in (1/2,1)$. A final section is devoted to provide examples of Poisson processes associated to some simple measures and to give integral conditions (non-discrete) on $\mu$ which are in some cases equivalent to the discrete versions used in the statements.

We finish this introduction recalling the Borel-Cantelli lemma, which is a central tool in this paper.  We refer to \cite{Bil} for a general source on probability theory. Given a sequence of events $A_k$ let $\limsup A_k=\{\omega:\omega\in A_k$ for infinitely many $k\}$. 

\begin{lemma}
Let $(A_k)_k$ be a sequence of events in a probability space. Then
\begin{enumerate} 
\item If $\sum \P(A_k)<\infty$, then $\P(\limsup A_k)=0$,
\item If the events $A_k$ are independent and $\sum \P(A_k)=\infty$, then $\P(\limsup A_k)=1$.
\end{enumerate}
\end{lemma}

{\bf Acknowledgements:} The authors would like to thank Joaquim Ortega-Cerd\`a for suggesting the consideration of Poisson processes and for helpful discussions.

%%%%%%%%%%%%%%%%%%%%%%%%%%%%%%%%%%%%%%%%%%%%%%%%%%%%%%%%%%%%%%%

\section{Proof of Theorems~\ref{thm:separation} and ~\ref{thm:Carleson}.}\label{theorems}

\subsection{Proof of Theorem~\ref{thm:separation}}
 Assume first that $\sum_{n,k}\mu_{n,k}^{M+1}<+\infty$, and define the events
\[
 A_{n,k}=\{X_{n,k}> M\}=\{X_{n,k}\ge M+1\}.
\]
Then
\[
 \P(A_{n,k})=1-\sum_{j=0}^M \P(X_{n,k}=j)=1-e^{-\mu_{n,k}} \bigl(\sum_{j=0}^M \frac{\mu_{n,k}^j}{j!}\bigr).
\]
By hypothesis $\lim\limits_n(\sup_k\mu_{n,k})= 0$, so we can use Taylor's formula
\begin{equation}\label{tay}
 1-e^{-x}(\sum_{j=0}^M\frac{x^j}{j!})=\frac{x^{M+1}}{(M+1)!}+o(x^{M+1})\qquad x\to 0
\end{equation}
to deduce that
\[
 \sum_{n,k} \P(A_{n,k})\lesssim \sum_{n,k}\frac{\mu_{n,k}^{M+1}}{(M+1)!}<+\infty.
\]
By the Borel-Cantelli lemma $X_{n,k}\leq M$ for all but at most a finite number of $T_{n,k}$. 

In principle this does not imply that $\Lambda_\mu$ can be split into $M$ separated sequences, because it 
might happen that points of two neighboring $T_{n,k}$ come arbitrarily close. This possibility is excluded by repeating
the above arguments to a new dyadic partition, made of shifted boxes $\tilde{T}_{n,k}$ having the ``lower vertices'' (those closer to $\T$) at the center of the $T_{n,k}$'s (see Figure \ref{Fig2} below); let
\[
 \tilde T_{n,k}=\Bigl\{z=re^{it}: \frac 32 2^{-(n+2)}<1-r\leq \frac 32 2^{-(n+1)}\, ;\ \frac{k+1/4}{2^n}\le \frac t{2\pi}<\frac{k+3/4}{2^n}\Bigr\}.
\]
Since each $\tilde T_{n,k}$ is included in the union of at most four $T_{m,j}$, we still have $\sum_{n,k} \tilde\mu_{n,k}^{M+1}<\infty$, and therefore, as before,  $\tilde X_{n,k}=N_{\tilde T_{n,k}}$ is at most $M$, except for maybe a finite number of indices $(n,k)$. This prevents that two adjacent $T_{n,k}$ have more than $M$ points getting arbitrarily close. In conclusion, for all but a finite number of indices $X_{n,k}\leq M$, hence the part of $\Lambda_\mu$ in these boxes can be split into $M$ separated sequences. Adding the remaining finite number of points  to any of these sequences may change the separation constant, but not the fact that they are separated.

\begin{figure}[H]
\centering
\begin{tikzpicture}[scale=0.6]
\centering
%\shade[top color=blue!15,bottom color=blue!20] (0,0) parabola bend (0,0) (2.2,2.42) parabola bend (1,2.7) (0,2)--(0,0);
\draw [ultra thick] (-4,0) -- (4,0);
\path [thin, blue,fill=blue!7] (-4,4)--(4,4)--(4,8)--(-4,8)--(-4,4);
\path [thin, blue,fill=blue!7] (4,1)--(6,1)--(6,2)--(4,2)--(4,1);
\draw [thin, red,fill=red!7] (2,3)--(6,3)--(6,6)--(2,6)--(2,3);
\draw [thin, red,fill=red!7] (-3,3)--(-3,1.5)--(-5,1.5)--(-5,3);
\draw [thin] (-9,0) -- (-4,0);
\draw [thin] (4,0) -- (9,0);
\draw [thin] (-4,8) -- (4,8);
\draw [thin] (-4,4) -- (4,4);
\draw [thin] (-4,2) -- (4,2);
\draw [thin, dashed] (-4,1) -- (4,1);
\draw [thin, dashed] (-3,0) -- (-3,1);
\draw [thin, dashed] (-1,0) -- (-1,1);
\draw [thin, dashed] (3,0) -- (3,1);
\draw [thin, dashed] (1,0) -- (1,1);

\draw [thin] (-4,0) -- (-4,8);
\draw [thin] (-8,0) -- (-8,4);
\draw [thin] (8,0) -- (8,4);
\draw [thin] (-8,4) -- (-4,4);

\draw [thin, dashed] (-9,4) -- (-8,4);
\draw [thin, dashed] (9,4) -- (8,4);
\draw [thin, dashed] (9,8) -- (4,8);
\draw [thin, dashed] (-9,8) -- (-4,8);
\draw [thin] (8,4) -- (4,4);
\draw [thin] (4,0) -- (4,8);
\draw [thin] (0,0) -- (0,4);
\draw [thin] (-2,0) -- (-2,2);
\draw [thin] (-8,2) -- (-4,2);
\draw [thin] (8,2) -- (4,2);
\draw [thin] (2,0) -- (2,2);
\draw [thin] (-6,0) -- (-6,2);
\draw [thin] (6,0) -- (6,2);
\draw [thin, dashed] (-8,1) -- (-4,1);
\draw [thin, dashed] (8,1) -- (4,1);
\draw [thin, dashed] (-7,0) -- (-7,1);
\draw [thin, dashed] (-5,0) -- (-5,1);
\draw [thin, dashed] (5,0) -- (5,1);
\draw [thin, dashed] (7,0) -- (7,1);

\draw [thin, red] (2,3)--(-2,3)--(-2,6)--(2,6);
\draw [thin, red] (-2,3)--(-6,3)--(-6,6)--(-2,6);
\draw [thin, red] (-8,3)--(-6,3);
\draw [thin, red, dashed] (-8,3)--(-9,3);
\draw [thin, red] (-8,6)--(-6,6);
\draw [thin, red, dashed] (-8,6)--(-9,6);

\draw [thin, red] (-5,3)--(-5,1.5)--(-7,1.5)--(-7,3)--(-5,3);
\draw [thin, red] (-7,3)--(-8,3);
\draw [thin, red, dashed] (-9,3)--(-8,3);
\draw [thin, red] (-8,1.5)--(8,1.5);
\draw [thin, red, dashed] (9,3)--(8,3);
\draw [thin, red, dashed] (9,1.5)--(8,1.5);
\draw [thin, red, dashed] (-9,1.5)--(-8,1.5);
\draw [thin, red, dashed] (-9,1.5)--(-9,3);
\draw [thin, red, dashed] (9,1.5)--(9,3);
\draw [thin, red] (-1,1.5)--(-1,3);
\draw [thin, red] (1,1.5)--(1,3);
\draw [thin, red] (3,1.5)--(3,3);
\draw [thin, red] (5,1.5)--(5,3);
\draw [thin, red] (7,1.5)--(7,3);

\draw [thin, red] (8,3)--(6,3);
\draw [thin, red, dashed] (8,3)--(9,3);
\draw [thin, red] (8,6)--(6,6);
\draw [thin, red, dashed] (8,6)--(9,6);

\node [below] at (0,0) {$I_{n,k}$};
%\node [right] at (4.2,4) {$Q_{n,k}$};
\node [above, blue] at (0,6.5) {$T_{n,k}$};
\node [ red] at (5,5) {$\tilde T_{n,k}$};

\end{tikzpicture}

\caption{Dyadic partitions: $\{T_{n,k}\}_{n,k}$ in blue, $\{\tilde T_{n,k}\}_{n,k}$ in red.}\label{Fig2}
\end{figure}
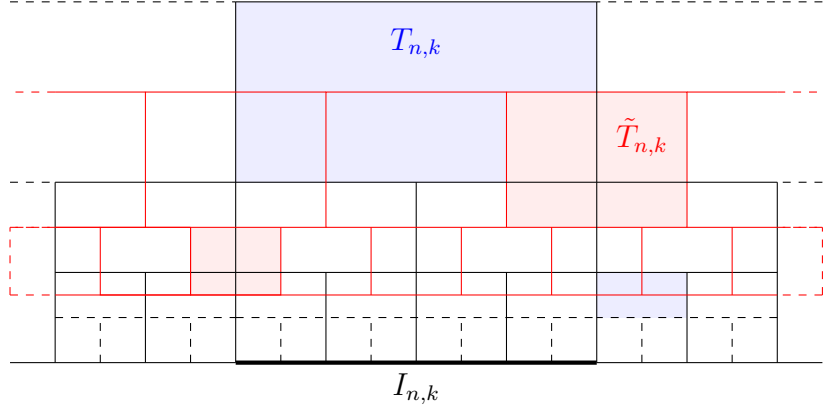

\medskip

Assume now that $\sum_{n,k}\mu_{n,k}^{M+1}=+\infty$. We shall prove that for
every $\delta_{l_0}=2^{-l_0}$, $l_0\in\N$,
\[ 
 \P\bigl(\text{$\Lambda$ union of $M$ $\delta_{l_0}$-separated sequences}\bigr)=0.
\]
Split each side of $T_{n,k}$ into $2^{l_0}$ segments of the same length. 
This defines a partition of $T_{n,k}$ in $2^{2l_0}$ small boxes of side length $2^{-n} 2^{-l_0}$, which we denote by
\[
 T_{n,k}^{l_0,j}\qquad j=1,\dots, 2^{2l_0}.
\]
Let $X_{n,k}^{l_0,j}=N_{T_{n,k}^{l_0,j}}$ denote the corresponding counting variable, which follows a Poisson law of parameter
$\mu_{n,k,l_0,j}=\mu(T_{n,k}^{l_0,j})$. 
 
It is enough to show that for any $l_0$, 
\[
 \P(X_{n,k}^{l_0,j}> M \text{ for infinitely many }n,k,j)=1.
\] 
By the second part of the Borel-Cantelli lemma, since the $X_{n,k}^{l_0,j}$ are independent, we shall be done as soon as we see that
\begin{equation}\label{BCSep}
 \sum_{n,k}\sum_{j=1}^{2^{2 l_0}} \P\bigl(X_{n,k}^{l_0,j}\ge M+1\bigr)=+\infty.
\end{equation}

For any Poisson variable $X$ of parameter $\lambda$, the probability
\[
 \P(X\ge M+1)=e^{-\lambda}\bigl(\sum_{m=M+1}^\infty \frac{\lambda^m}{m!}\bigr)=1-e^{-\lambda}\bigl(\sum_{m=0}^M \frac{\lambda^m}{m!}\bigr)
\]
increases in $\lambda$. Hence there is no restriction in assuming that $0\leq\mu_{n,k,l_0,j}\leq \mu_{n,k}\leq 1/2$ for all $n,k,j$. Then we can use Taylor's formula \eqref{tay} to deduce that
\[
 \P\bigl(X_{n,k}^{l_0,j}\ge M+1\bigr)\simeq \frac{\mu_{n,k,l_0,j}^{M+1}}{(M+1)!}.,
\]
and therefore \eqref{BCSep} is equivalent to 
\[
 %\sum_{n,k}\sum_{j=1}^{2^{2l_0}} P\bigl(X_{n,k}^{l_0,j}\ge M+1\bigr)\simeq 
 \sum_{n,k}\sum_{j=1}^{2^{2l_0}} \mu_{n,k,l_0,j}^{M+1}=+\infty.
\]
That this sum in infinite is just a consequence of the hypothesis and the elementary estimate
\begin{align*}
 \mu_{n,k}^{M+1}=\Bigl(\sum_{j=1}^{2^{2l_0}} \mu_{n,k,l_0,j}\Bigr)^{M+1}
 \leq 2^{2l_0 (M+1)} \sum_{j=1}^{2^{2l_0}} \mu_{n,k,l_0,j}^{M+1}.
\end{align*}

\subsection{Proof of Theorem~\ref{thm:Carleson}}
 
(a) 
Assume first that $\sum_n \mu_{n,k}^{\gamma}<+\infty$ for some $\gamma>1$.   %It could be tempting to use Theorem \ref{thm:separated}. There are however two observations to made. First, that theorem does not cover the values $\gamma\in (1,2)$, and, more importantly, a separated sequence is a priori not   %This part of the proof follows the improvement given in \cite{CHKW}*{Theorem 1.1} of the scheme proposed by R. Rudowicz in \cite{Ru} for the probabilistic model with fixed radii and uniform random angles.
It is enough to check the Carleson condition
\[
 \sum_{\lambda\in Q(I)} (1-|\lambda|)\leq C |I|
\]
on the dyadic intervals $I_{n,k}$. Let $Q_{n,k}=Q(I_{n,k})$. Decomposing the sum on the different layers $A_m$, 
it is enough to show that almost surely there exists $C>0$ such that for all $n\geq 0$, $k=0,\dots, 2^{n-1}$
\[
\sum_{\lambda\in Q_{n,k}} (1-|\lambda|)\simeq \sum_{m\geq n} \sum_{j: T_{m,j}\subset Q_{n,k}} 2^{-m} X_{m,j}\leq C 2^{-n} .
\]
This is equivalent to
\begin{equation}\label{eq:Carl-dyadic}
 \sup_{n,k}\ 2^n\sum_{m\geq n} \sum_{j: T_{m,j}\subset Q_{n,k}} 2^{-m} X_{m,j}<\infty
\end{equation}

Denote 
\[
 X_{n,m,k}=N_{Q_{n,k}\cap A_m}=\#(\Lambda\cap Q_{n,k}\cap A_m)=\sum_{j: T_{m,j}\subset Q_{n,k}} X_{m,j},
\]
which is a Poisson variable of parameter
\[
 \mu_{n,m,k}=\mu(Q_{n,k}\cap A_m)=\sum_{j: T_{m,j}\subset Q_{n,k}} \mu_{m,j}.
\]
Set
\[
 Y_{n,k}=2^n\sum_{m\ge n}2^{-m}X_{n,k,m}=\sum_{m\ge n}2^{n-m}X_{n,k,m},
\]
so that \eqref{eq:Carl-dyadic} becomes $\sup_{n,k} Y_{n,k}<+\infty$.

Let $A>0$ be a big constant to be fixed later on. Again by the Borel-Cantelli Lemma, it is enough to show that
\begin{equation}\label{est:ynk}
\sum_{n,k} \P\bigl(Y_{n,k}>A\bigr)<+\infty,
\end{equation}
since then $Y_{n,k}\leq A$ for all but maybe a finite number of $n,k$; in particular $\sup_{n,k} Y_{n,k}<\infty$.

The first step of the following reasoning is an adaptation to the Poisson process of the proof given in \cite{CHKW}*{Theorem 1.1} and which allowed to improve the result on Carleson sequences for the probabilistic model with fixed radii and random arguments. However, while in the original proof the Carleson boxes $Q_{n,k}$ are decomposed into layers $Q_{n,k}\cap A_m$ ($m\ge n$), in this new situation (as well as for (b)), Carleson boxes are decomposed into top-halves $T_{m,j}\subset Q_{n,k}$, which requires more delicate arguments to reach the convergence needed in the Borel-Cantelli lemma.
%Still, non-trivial adjustments have to be introduced here --- as well as in the proof of b) --- in order to take into account the fact that the Poisson process is controlled in upper halves of Carleson boxes, and that when decompo

Recall that the probability generating
function of a Poisson variable $X$ of parameter $\lambda$ is $\E(s^X)=e^{\lambda(s-1)}$. 
By the independence of the different $X_{n,k,m}$, $m\ge n$, 
\[
 \E(s^{Y_{n,k}})=\prod_{m\ge n}\E((s^{2^{n-m}})^{X_{n,m,k}})
 =\prod_{m\ge n}e^{\mu_{n,m,k}(s^{2^{n-m}}-1)}.
\]
Thus for any $s>1$, by Markov's inequality
\[
 \P(Y_{n,k}>A)=\P(s^{Y_{n,k}}>s^A)
 \le \frac{1}{s^A}\E(s^{Y_{n,k}})=\frac{1}{s^A}\prod_{m\ge n}e^{\mu_{n,m,k}(s^{2^{n-m}}-1)}.
\]
Using the estimate $x(a^{1/x}-1)\le a$, for $a,x>1$, with $a=s$ and $x=2^{m-n}$,
\begin{align*}
 \log \P(Y_{n,k}>A)&\le -A\log s+\sum_{m\ge n}(s^{2^{n-m}}-1)\, \mu_{n,m,k}\\
 &\le -A\log s+\sum_{m\ge n} s 2^{n-m} \, \mu_{n,m,k}\\
 &=-A\log s+s\sum_{m\ge n} 2^{n-m}
 \sum_{j:T_{m,j}\subset Q_{n,k}}\mu_{m,j}.
\end{align*}
We want to optimize this estimate for $s>1$. Set
\[
 B_{n,k}=\sum_{m\ge n} 2^{-(m-n)}
 \sum_{j:T_{m,j}\subset Q_{n,k}}\mu_{m,j}
 %=\sum_{l\ge 0}2^{-l}\sum_{j:T_{n+l,j}\subset Q_{n,k}} \mu_{n+l,j}
\]
and define
\[
 \phi (s)=-A\log s+s B_{n,k}.
\]
Let us observe first that the $B_{n,k}$ are uniformly bounded (they actually tend to 0). Indeed, let $\beta$ denote the conjugate exponent of $\gamma$ ($\frac 1{\gamma}+\frac 1{\beta}=1$). Since for $m\geq n$ there are $2^{m-n}$ boxes $T_{m,j}$ in $Q_{n,k}$, by H\"older's inequality on the sum in the index $j$ we deduce that
\begin{align*}
 B_{n,k}&\le \sum_{m\ge n}2^{-(m-n)}\Bigl(\sum_{j:T_{m,j}\subset Q_{n,k}} \mu_{m,j}^{\gamma}
 \Bigr)^{1/\gamma}2^{(m-n)/\beta}\\
 &=\sum_{m\ge n}2^{-(m-n)/\gamma}\Bigl(\sum_{j:T_{m,j}\subset Q_{n,k}} \mu_{m,j}^{\gamma}
 \Bigr)^{1/\gamma}<+\infty.
\end{align*}
Taking $A$ big enough we see that the minimum of $\phi$ is attained at $s_0=A/B_{n,k}>1$. Hence
\[
 \log \P(Y_{n,k}>A) \le \phi(s_0)=-A\log\frac{A}{B_{n,k}}+A.
\]
Therefore
\[
  \P\bigl(Y_{n,k}>A\bigr)\le \left(\frac{B_{n,k}}{A}\right)^Ae^A,
\]
and
\[
 \sum_{n,k} \P(Y_{n,k}>A)
 \le \left(\frac{e}{A}\right)^A\sum_{n,k}B_{n,k}^A.
\]
The estimate on $B_{n,k}$ obtained previously is not enough to prove that this last sum converges.
In order to obtain a better estimate take $p>1$, to be chosen later on, its conjugate exponent $q$ (i.e. $\frac 1p+\frac 1q=1$),
and apply H\"older's inequality in the following way:
\begin{align*}
 B_{n,k}&=\sum_{\substack{m\geq n\\j:T_{m,j}\subset Q_{n,k}}}  2^{-(m-n)}\mu_{m,j}
 =2^n \sum_{\substack{m\geq n\\j:T_{m,j}\subset Q_{n,k}}}  2^{-\frac mp}2^{-\frac mq}\mu_{m,j}\\
&\le 2^n\Bigl(\sum_{\substack{m\geq n\\j:T_{m,j}\subset Q_{n,k}}} 2^{-\frac{m\beta}p}\Bigr)^{1/\beta}
\times\Bigl(\sum_{\substack{m\geq n\\j:T_{m,j}\subset Q_{n,k}}} 2^{-\frac{m\gamma}q}\mu^{\gamma}_{m,j}\Bigr)^{1/\gamma}.
\end{align*}

Choose now $p$ so that $1<p<\beta$; then 
%Since the layer $A_m$ in $Q_{n,k}$ contains $2^{m-n}$ top halves $T_{m,j}$,
\begin{align*}
\sum_{\substack{m\geq n\\j:T_{m,j}\subset Q_{n,k}}}   2^{-\frac{m\beta}p}
 &=\sum_{m=n}^{\infty}  2^{-\frac{m\beta}p}  \ 2^{m-n}
 =2^{-n}\sum_{m=n}^{\infty}2^{-m(\frac{\beta}p-1)}
 \simeq 2^{-n}2^{-n(\frac{\beta}p-1)}=2^{-n\frac{\beta}p}.
\end{align*}
Thus, from the above estimate, 
%since $2^n 2^{-n\frac{1}p}=2^{\frac nq}$,
\[
 B_{n,k}\le 2^{\frac nq}
 \Bigl(\sum_{\substack{m\geq n\\j:T_{m,j}\subset Q_{n,k}}}    2^{-\frac{m\gamma}q}\ \mu^{\gamma}_{m,j} \Bigr)^{1/\gamma}.
\]
Choosing $A=\gamma$ yields 
\[
\sum_{n,k}B_{n,k}^\gamma \le \sum_{n,k}  2^{\frac{n\gamma}q}
 \sum_{\substack{m\geq n\\j:T_{m,j}\subset Q_{n,k}}}  2^{-\frac{m\gamma}q}\ \mu^{\gamma}_{m,j} .
\]
We now apply Fubini's theorem to exchange the sums. The important observation here is that
each $T_{m,j}$ has only one ancestor at each level $n\le m$ (i.e, one $T_{n,k}$ containing
$T_{m,j}$). Hence
\begin{align*}
\sum_{n,k} B_{n,k}^{\gamma}&\le\sum_{m,j} 2^{-\frac{m\gamma}q}\ \mu^{\gamma}_{m,j}
\sum_{\substack{n\leq m\\ k : Q_{n,k} \supseteq T_{m,j}}}  2^{\frac{n\gamma}p}
 =\sum_{m,j}  2^{-\frac{m\gamma}q}\ \mu^{\gamma}_{m,j}  \sum_{n\le m} 2^{\frac {n\gamma} q}\\
 &\le 2\sum_{m,j} 2^{-\frac{m\gamma}q}   \mu^{\gamma}_{m,j}\ 2^{\frac{m\gamma}q}
 =2\sum_{m,j} \mu^{\gamma}_{m,j}.
\end{align*}
This finishes the proof of \eqref{est:ynk}, hence of this part of the theorem.

\medskip

Let us now assume that $\sum_{n,k}\mu_{n,k}^{\gamma}
=+\infty$ for every $\gamma>1$. Suppose $M\ge 1$ is an integer. Since the sum diverges for $\gamma=M+1$, Theorem~\ref{thm:separation} implies that the sequence $\Lambda_{\mu}$ is almost surely not a union of $M$ separated sequences. In particular, a.s.\ there is $\lambda_0\in\Lambda_{\mu}$ such that $D_{\lambda_0}=\{z\in \D:\rho(\lambda_0,z)<1/2\}$ contains at least $M+1$ points of $\Lambda_{\mu}$. Then, letting $I_{\lambda_0}$ be the interval centered at $\lambda_0/|\lambda_0|$ with length $1-|\lambda_0|$, we have $\sum_{\lambda\in Q(I_{\lambda_0})}(1-|\lambda|)\gtrsim M |I_{\lambda_0}|$, where the underlying constant does not depend on $M$ or $\lambda_0$. This being true for every integer $M\ge 1$, the sequence cannot be 1-Carleson.
%Since the event that $\mu_{\Lambda}$ is Carleson is a tail event, Kolmogorov's 0-1 law
%allows to show the equivalence.

\medskip

(b) 
Proceeding as in the first implication of (a) we see that it is enough to prove that almost surely
\begin{equation}\label{eq:carl-alpha}
\sup_{n,k} Y_{n,k}<+\infty\ ,
\end{equation}
where now
\begin{equation}\label{ynk-alpha}
 \quad Y_{n,k}=2^{n\alpha}\sum_{m\geq n} 2^{-m\alpha} \sum_{j: T_{m,j}\subset Q_{n,k}} X_{m,j}.
\end{equation}
The same estimates as in (a) based on the probability generating function yield, for $s>1$,
\[
 \log \P\bigl(Y_{n,k}\geq A\bigr)\leq \phi(s)=-A\log s+ B_{n,k},
\]
where now
\[
 B_{n,k}=\sum_{m\geq n}  2^{-(m-n)\alpha} \sum_{j: T_{m,j}\subset Q_{n,k}} \mu_{m,j}.
\]
As in (a), the hypotheses imply that $B_{n,k}$ is uniformly bounded: letting $\beta$ denote the conjugate exponent to $\gamma$ ($\frac 1{\gamma}+\frac 1{\beta}=1$) and noticing that $\alpha-1/\beta=1/\gamma-(1-\alpha)>0$,
\begin{align*}
 B_{n,k}&\leq \sum_{m\geq n}  2^{-(m-n)\alpha} \Bigl(\sum_{j: T_{m,j}\subset Q_{n,k}} \mu_{m,j}^\gamma\Bigr)^{1/\gamma} \ 2^{(m-n)/\beta} \\
 &= \sum_{m\geq n}  2^{-(m-n)(\alpha-1/\beta)}  \Bigl(\sum_{j: T_{m,j}\subset Q_{n,k}} \mu_{m,j}^\gamma\Bigr)^{1/\gamma}.
\end{align*}
Therefore, optimizing the estimate for $s>1$ exactly as we did in (a), we obtain $\P\bigl(Y_{n,k}\geq A\bigr)\lesssim B_{n,k}^A$, and we are lead to prove that for some $A>0$
\begin{equation}\label{sum-c-alpha}
 \sum_{n,k} \P\bigl(Y_{n,k}\geq A\bigr)\lesssim \sum_{n,k} B_{n,k}^A<\infty.
\end{equation}
Again, we introduce an auxiliary weight $p$ -- to be determined later -- and its conjugate exponent $q$. Split $2^{-m\alpha}=2^{-\frac{m\alpha}p}2^{-\frac{m\alpha}q}$ and use H\"older's inequality to obtain
\begin{align*}
 B_{n,k}&
 %\sum_{\substack{m\geq n\\j:T_{m,j}\subset Q_{n,k}}}  2^{-(m-n)\alpha}\mu_{m,j}
 %=2^{n\alpha} \sum_{\substack{m\geq n\\j:T_{m,j}\subset Q_{n,k}}}  2^{-\frac m\alpha p}2^{-\frac m\alpha q}\mu_{m,j}\\
\le 2^{n\alpha}\Bigl(\sum_{\substack{m\geq n\\j:T_{m,j}\subset Q_{n,k}}} 2^{-\frac{m\alpha\beta}p}\Bigr)^{1/\beta}
\times\Bigl(\sum_{\substack{m\geq n\\j:T_{m,j}\subset Q_{n,k}}} 2^{-\frac{m\alpha\gamma}q}\mu^{\gamma}_{m,j}\Bigr)^{1/\gamma}.
\end{align*}
The first sum is finite:  since by hypothesis $\alpha\beta=\frac{\alpha \gamma}{\gamma-1}>1$ there exists 
$1<p<\frac{\alpha \gamma}{\gamma-1}$ and,
\begin{align*}
 \sum_{\substack{m\geq n\\j:T_{m,j}\subset Q_{n,k}}} 2^{-\frac{m\alpha\beta}p}&=\sum_{m\geq n} 2^{-\frac{m\alpha\beta}p} 2^{m-n}
 =2^{-n} \sum_{m\geq n} 2^{-m(\frac{\alpha\beta}p-1)}\simeq  2^{-n\frac{\alpha\beta}p}.
\end{align*}
This implies that
\[
 B_{n,k}^\gamma\lesssim 2^{n\alpha\frac{\gamma}q}  \sum_{\substack{m\geq n\\j:T_{m,j}\subset Q_{n,k}}} 2^{-\frac{m\alpha\gamma}q}\mu^{\gamma}_{m,j} 
\]
and we can conlude the proof of \eqref{sum-c-alpha} as before:
\begin{align*}
 \sum_{n,k} B_{n,k}^\gamma &\lesssim\sum_{n,k} 2^{n\alpha\frac{\gamma}q}  \sum_{\substack{m\geq n\\j:T_{m,j}\subset Q_{n,k}}} 
 2^{-\frac{m\alpha\gamma}q}  \mu^{\gamma}_{m,j} =\sum_{m,j} \mu^{\gamma}_{m,j} 2^{-m\alpha\frac{\gamma}q} \sum_{\substack{n\leq m\\ k: Q_{n,k} \supseteq T_{m,j} }}   2^{n\alpha\frac{\gamma}q}\\
 &=\sum_{m,j} \mu^{\gamma}_{m,j} 2^{-m\alpha\frac{\gamma}q} \sum_{n\leq m} 2^{-n\alpha\frac{\gamma}q}\simeq \sum_{m,j} \mu^{\gamma}_{m,j} <+\infty. %\quad \Box
\end{align*}

\medskip

(c)  Here we  give a measure $\mu$ for which $\sum_{n,k}\mu_{n,k}^{1/(1-\alpha)}<+\infty$ but $\P(\Lambda_\mu\ \textrm{is $\alpha$-Carleson})=0$. Let  
\[
d\mu(z)= \frac{dm(z)}{(1-|z|^2)^{1+\alpha} \log\bigl(\frac e{1-|z|^2}\bigr)},
\]
which is the measure $\mu=\mu(1+\alpha,1)$ given in the family of examples of Section~\ref{Examples}.
By a simple computation (see \eqref{ex:radial})
\[
 \mu_{n,k}\simeq \frac{2^{-n(1-\alpha)}}n\qquad n\geq 1,\ k=0, \dots, 2^n-1
\]
and therefore, since $k$ ranges over $2^n$ terms, 
\begin{align*}
 \sum_{n,k}\mu_{n,k}^{1/(1-\alpha)}\simeq \sum_{n\geq 1} \frac{1}{n^{1/(1-\alpha)}}<+\infty.
\end{align*}

On the other hand, letting $Y_{n,k}$ be as in the proof of part (b) (see \eqref{ynk-alpha})
we get 
\beqa
 \E(Y_{n,k})&=&2^{n\alpha}
 \sum_{m\ge n}2^{-m\alpha}\sum_{j:T_{m,j}\subset Q_{n,k}}\mu_{m,j}
\simeq 2^{n\alpha}
 \sum_{m\ge n}2^{-m\alpha} 2^{m-n}\frac{2^{-(1-\alpha)m}}{n}\\
 &=&2^{-(1-\alpha)n}\sum_{m\ge n}\frac{1}{n}=+\infty
\eeqa
Thus the expected weight of any single Carleson window $Q_{n,k}$ is infinite and  $\Lambda_{\mu}$ cannot be $\alpha$-Carleson.

\medskip

(d) One could think of considering a divergent
series $\sum_{n,k}\mu_{n,k}^{\gamma}=+\infty$ such that $\sum_{n,k}\mu_{n,k}^{\gamma'}
<+\infty$ for every $\gamma'>\gamma$, and then apply (b), showing that $\Lambda_{\mu}$
is $\alpha$-Carleson when $\gamma'<\frac{1}{1-\alpha}$, i.e. when $\alpha>1-\frac{1}{\gamma'}=\frac{\gamma'-1}{\gamma'}$. However, this does not yield the whole range $\alpha\in (0,1)$ for a fixed measure, as required by the statement.

In order to construct an example working for all $\alpha\in (0,1)$, we pick a measure $\mu$ supported in a Stolz angle of vertex 1, i.e. let, for $n\geq 1$,
\[
 \mu_{n,k}=\begin{cases}
 \dfrac{1}{n^{1/\gamma}}&\textrm{if }k=0\\
\quad 0 &\textrm{if } k>1. 
 \end{cases}
\]
(We could equivalently take the measure $\tau(2,1/\gamma)$ given in Subsection~\ref{Examples}, Example 3).
Then
\begin{equation}\label{SumDiv}
 \sum_{n,k}\mu_{n,k}^{\gamma}=\sum_n \frac{1}{n}=\infty
\end{equation}
but for every $\gamma'>\gamma$,
\[
\sum_{n,k}\mu_{n,k}^{\gamma'}=\sum_n \frac{1}{n^{\gamma'/\gamma}}<+\infty.
\]

To prove that $\Lambda_\mu$ is almost surely $\alpha$-Carleson 
we will argue as before. Set $Y_{n,k}$ as in the proof of (b) (see \eqref{ynk-alpha}) and follow the same steps to prove that
\[
 \P(Y_{n,k}\ge A)\lesssim B_{n,k}^A,
\]
where
\[
 B_{n,k}=\sum_{m\ge n}
\sum_{j:T_{m,j}\subset Q_{n,k}} \mu_{m,j}2^{-(m-n)\alpha}.
\]
By construction $B_{n,k}=0$ for all $k>0$. On the other hand
\[
 B_{n,0}=2^{n\alpha}\sum_{m\ge n}2^{-m\alpha}\mu_{m,0}
 =2^{n\alpha}\sum_{m\ge n}\frac{2^{-m\alpha}}{m^{1/\gamma}}
\le \frac{1}{n^{1/\gamma}}.
\] 
(Observe that this last expression is independent of $\alpha$.)
Hence 
\[
 \sum_{n,k}B_{n,k}^{\gamma'}=\sum_nB_{n,0}^{\gamma'}
\le \sum_n \frac{1}{n^{\gamma'/\gamma}}<+\infty,
\]
and as in the proof of (b) the Borel-Cantelli lemma allows to conclude that $\Lambda$ is
almost surely $\alpha$-Carleson.

%{\bf Remarque:} It is just an afterthought of the above example (and maybe not surprising), that for every $\gamma>1$ there is $\mu$ such that $\sum_{n,k}\mu_{n,k}^{\gamma}<+\infty$ and $P(\Lambda_{\mu} \alpha-$Carleson$)=1$. Indeed, pick $\mu_{n,0}=1/(n+1)^{1/\gamma'}$ for $\gamma'<\gamma$, then $\sum_{\mu_{n,k}}\mu_{n,k}^{\gamma}<+\infty$, and the above reasons show that $\Lambda_{\mu}$ is almost surely $\alpha$-Carleson.
 %%%%%%%%%%%%%%%%%%%%%%%%%%%%%%%%%%%%%%%%%%%%%%%%%%%%%%%%%%%%%%

\section{Random interpolating sequences}\label{int_h}

In this section we discuss several consequences of Theorems~\ref{thm:separation} and ~\ref{thm:Carleson} on random interpolating sequences $\Lambda_\mu$ for various spaces of holomorphic functions in $\D$. The results are rather straightforward consequences of the aforementioned theorems and  the known conditions for such sequences. 

%properties of sequences related to interpolation in different spaces of holomorphic functions. Such properties are for instance separation, Carleson measures, zero-sequences. If one remembers that in a large class of spaces of holomorphic fonctions containing Hardy and Dirichlet spaces separation and Carleson measures characterize interpolation (see \cite{AHMR}), it is interesting to recall that in \cite{CHKW} it was shown that in certain weighted Dirichlet spaces almost sure zero sequences are almost surely interpolating, while in others it is the almost sure separation which characterizes interpolation almost surely. We do not claim treating all these properties for all the spaces here (for instance zero sequences in Dirichlet spaces obtained from Point processes are very challenging). 

%In this section we would like to focus on those results that can be quite immediately deduced from our Theorems ~\ref{thm:separation} and ~\ref{thm:Carleson}.
%We mention that the separation condition is already established by Theorem~\ref{thm:separation} and the Carleson measure condition in the Hardy space by Theorem~\ref{thm:Carleson}(a). 
% we state several corollaries of Theorems ~\ref{thm:separation} and ~\ref{thm:Carleson} on interpolating sequences for various spaces of holomorphic functions in $\D$.

\subsection{Hardy (and Bergman) spaces}\label{HardyBergman}
%Let us start completing the different results in the Hardy space. Indeed, for Poisson point processes we can completely characterize the four properties: separation, Carleson measure condition, interpolation and zero sequences. The first two properties, separation and Carleson measure condition ($1$-Carleson), are already established in Theorems ~\ref{thm:separation} and ~\ref{thm:Carleson}(a).

In this section we completely characterize the measures $\mu$ for which the associated Poisson process $\Lambda_\mu$ is almost surely an interpolating sequence for the Hardy spaces.

Recall that a sequence $\Lambda=\{\lambda_n\}_{n\in\N}\subset\D$ is interpolating for
\[
 H^\infty=\bigl\{f\in H(\D) : \|f\|_\infty=\sup_{z\in D} |f(z)|<\infty\bigr\}
\]
whenever for every sequence of bounded values $\{w_n\}_{n\in\N}\subset\C$ there exists $f\in H^\infty$ such that $f(\lambda_n)=w_n$, $n\in\N$. According to a famous theorem by L. Carleson, $\Lambda$ is $H^\infty$-interpolating if and only if it is separated and 1-Carleson \cite{Ca}. This characterization extends to all Hardy spaces
\[
 H^p=\Bigl\{f\in H(\D) : \|f\|_p=\sup_{r<1}\Bigl(\int_0^{2\pi} |f(re^{it})|^p\, \frac{dt}{2\pi}\Bigr)^{1/p}<+\infty\Bigr\}\qquad 0<p<\infty,
\]
for which the interpolation problem is defined in a similar manner (the data $w_n$ to be interpolated should satisfy $\sum_n(1-|\lambda_n|^2)|w_n|^p<+\infty$, see e.g. \cite{Du}*{Chapter 9}). %We remind also that more generally, Carleson measures are Borel measures defined on $\D$ such that $H^p\subset L^p(\mu)$, and that Carleson characterized these measures by the condition $\mu(Q(I))\le C|I|$ which reduces to the $1$-Carleson property in case $\mu=\sum_{\lambda}(1-|\lambda|)\delta_{\lambda}$.

The separation condition given in Theorem~\ref{thm:separation} implies immediately that $\Lambda_\mu$ is 1-Carleson, by Theorem~\ref{thm:Carleson}, hence the following result follows.

\begin{theorem}\label{thm:Hardy}
 Let $\Lambda_\mu$ be the Poisson process associated to a positive, $\sigma$-finite, locally finite measure $\mu$. Then, for any $0<p\leq \infty$,
  \[
   \P\bigl(\Lambda_\mu\ \textrm{is $H^p$-interpolating}\bigr) =
   \begin{cases}
   1\quad &\textrm{if $\ \displaystyle\sum\limits_{n,k}\mu_{n,k}^{2}<\infty$} \\
   0\quad &\textrm{if $\ \displaystyle\sum\limits_{n,k}\mu_{n,k}^{2}=\infty$}.
  \end{cases}
   \]
\end{theorem}

To complete the picture we discuss zero sequences $\Lambda$ for $H^p$, $0<p\leq \infty$. These are deterministically characterized by the Blaschke condition $\sum_{\lambda\in\Lambda}(1-|\lambda|)<\infty$.
%, and it turns out that it is quite easy to give conditions so that $\Lambda_\mu$ is almost surely such a zero set. 
Noticing that $\{\sum_{\lambda\in\Lambda_\mu}(1-|\lambda|)<\infty\}$ is a tail event and
using Kolmogorov's 0-1 law we get:

\begin{proposition}\label{prop:Blaschke}
 Let $\Lambda_\mu$ be the Poisson process associated to a positive, $\sigma$-finite, locally finite measure $\mu$. Then, for any $0<p\leq \infty$,
  \[
   \P\bigl(\Lambda_\mu\ \textrm{is a zero set for $H^p$}\bigr) =
   \begin{cases}
   1\quad &\textrm{if $\ \displaystyle\sum\limits_{n,k}2^{-n}\mu_{n,k}<\infty$} \\
   0\quad &\textrm{if $\ \displaystyle\sum\limits_{n,k}2^{-n}\mu_{n,k}=\infty$}.
  \end{cases}
   \]
\end{proposition}

Observe that the condition is just
\[
 \E\bigl[\sum_{\lambda\in\Lambda_\mu}(1-|\lambda|)\bigr]=\E\bigl[\sum_{n,k}\sum_{\lambda\in T_{n,k}}(1-|\lambda|)\bigr]
 \simeq \sum\limits_{n,k}2^{-n}\E\bigl[X_{n,k}\bigr] =\sum\limits_{n,k}2^{-n}\mu_{n,k}<\infty.
\]
Observe also that $\sum_{k=0}^{2^n-1}\mu_{n,k}=\mu(A_n)$ for all $n\in\N$, hence
\[
 \sum\limits_{n,k}2^{-n}\mu_{n,k}=\sum_n 2^{-n}\mu(A_n).
\]

\begin{proof}[Proof of Proposition~\ref{prop:Blaschke}]
Denote $X_n=N_{A_n}=\sum_{k=0}^{2^n-1} X_{n,k}$ and denote $\mu_n=\E[X_n]=\mu(A_n)$.

Assume first that $\sum_{n} 2^{-n} \mu_n<+\infty$. Set $Y=\sum_n 2^{-n} X_n$ and observe that, by the independence of the different $X_n$,
\begin{align*}
\E[Y]&=\sum_n 2^{-n} \mu_n<+\infty\quad , \quad
&& \Var(Y)=\sum_n 2^{-2n}\mu_n<+\infty.
\end{align*}
Then, by Markov's inequality
\[
 \P(Y\ge 2\E(Y))\le \frac{1}{2}.
\]
Since $\{Y=\infty\}$ is a tail event, Kolmogorov's 0-1 law implies that $\P(Y=+\infty)=0$, and in particular
the Blaschke sum is finite almost surely.

Assume now that $\sum_{n} 2^{-n} \mu_n=+\infty$. Split the sum in two parts:
\[
 \sum_n 2^{-n}\mu_n=\sum_{n:\mu_n\le 2^n/n^2}2^{-n}\mu_n+ 
 \sum_{n:\mu_n> 2^n/n^2}2^{-n}\mu_n.
\]
It is enough to consider the second sum, since the first one obviously converges. Since $\Var[X_n]=\mu_n$, Chebyshev's inequality yields, 
%\begin{eqnarray*}
\[
\P(X_n\le \frac{1}{2}\mu_n) =\P(X_n\le \mu_n-\frac{\mu_n}{2})\leq \P(|X_n- \mu_n|\geq \frac{\mu_n}{2})
 \le \frac{4}{\mu_n}.
\]
%\end{eqnarray*}
Hence
\[
 \sum_{n:\mu_n> 2^n/n^2} \P(X_n\le \frac{1}{2}\mu_n) 
 \le \sum_{n:\mu_n> 2^n/n^2}\frac{4}{\mu_n}\le \sum_{n:\mu_n> 2^n/n^2}\frac{4n^2}{2^n}
 <+\infty.
\]
Now, by the Borel-Cantelli lemma,  $X_n>\frac{1}{2}\mu_n$ for all but maybe a finite number of the $n$
with $\mu_n>2^n/n^2$; hence
\[
 \sum_{n:\mu_n> 2^n/n^2} 2^{-n}X_n
 \succsim\frac{1}{2}\sum_{n:\mu_n> 2^n/n^2}2^{-n}\mu_n,
\]
which diverges, by hypothesis.
\end{proof}

\subsubsection{Remark. Interpolation in Bergman spaces}
 Interpolating sequences $\Lambda$ for the (weighted) Bergman spaces
 \[
  B_\alpha^p=\Bigl\{f\in H(\D) : \|f\|_{\alpha,p}^p=\int_{\D} |f(z)|^p (1-|z|^2)^{\alpha p-1} dm(z)<\infty\Bigr\},
 \]
with $0<\alpha$, $0<p\leq \infty$ are characterized by the separation together with the upper density condition
\[
 D_+(\Lambda):=\limsup_{r\to 1^-} \sup_{z\in\D} \frac{\sum\limits_{1/2<\rho(z,\lambda)\leq r}\log\frac 1{\rho(z,\lambda)}}{\log(\frac 1{1-r})}<\alpha
\]
(see \cite{Se2} and \cite{HKZ}*{Chapter 5} for both the definitions and the results).

Since every $1$-Carleson sequence has density $D_+(\Lambda)=0$, the same conditions of Theorem~\ref{thm:Hardy} also characterize a.s. Bergman   interpolating sequences, regardless of the indices $\alpha$ and $p$. Again, because of the big fluctuations of the Poisson process, the conditions required to have separation a.s. are so strong that they can only produce sequences of zero upper density.

Another indication of the big fluctuations of the Poisson process is the following. For 
the invariant measure $d\nu(z)=\frac{dm(z)}{(1-|z|^2)^2}$, which obviously satisfies $\nu_{n,k}\simeq 1$ for all $n$, $k$, it is not difficult to see that almost surely,
\[
  D_+(\Lambda_{\nu})=+\infty\qquad\textrm{and}\quad D_-(\Lambda_{\nu}):=\liminf_{r\to 1^-} \inf_{z\in\D} \frac{\sum\limits_{1/2<\rho(z,\lambda)\leq r}\log\frac 1{\rho(z,\lambda)}}{\log(\frac 1{1-r})}=0.
\]
%\tcr{cela m\'eriterait plus de d\'etails?}
Therefore there are way too many points for $\Lambda_\nu$ to be interpolating for any $B_\alpha^p$, but there are too few for it to be sampling, since these sets must have strictly positive lower density $D_-(\Lambda)$ (see \cite{HKZ}*{Chapter 5}).

\subsection{Interpolation in the Bloch space}

We consider now interpolation in the Bloch space $\mathcal B$, consisting of functions $f$ holomorphic in $\D$ such that
\[
 \|f\|_{\mathcal B}:=|f(0)|+\sup_{z\in\D}|f'(z)|(1-|z|^2)<+\infty.
\]
Since Bloch functions satisfy the Lipschitz condition $|f(z)-f(w)|\leq \|f\|_{\mathcal B}\, \delta(z,w)$, where
$\delta(z,w)=\frac{1}{2}\log\frac{1+\rho(z,w)}{1-\rho(z,w)}$ denotes the hyperbolic distance, A. Nicolau and B. B\o e   defined interpolating sequences for $\mathcal B$ as those $\Lambda=\{\lambda_n\}_{n\in\N}$ such that for every sequence of values $\{v_n\}_{n\in\N}$ with $\sup\limits_{n\neq m}\frac{|v_n-v_m|}{\delta(\lambda_n,\lambda_m)}<\infty$ there exists $f\in\mathcal B$ with $f(\lambda_n)=v_n$, $n\in\N$ \cite{BN}. 

\begin{theorem*}[\cite{BN}*{pag.172}, \cite{S}*{Theorem 7}]\label{thmBN}
A sequence $\Lambda$ of distinct points in $\D$ is an interpolating sequence for
$\mathcal B$ if and only if:
\begin{itemize}
\item[(a)] $\Lambda$ can be expressed as the union of at most two separated sequences,

\item[(b)] for some $0<\gamma<1$ and $C>0$,
\[
 \#\bigl\{\lambda\in\Lambda : \rho(z,\lambda)<r\bigr\}\le \frac{C}{(1-r)^{\gamma}}
\]
independently on $z\in\D$.
\end{itemize}
\end{theorem*}

As explained in \cite{BN}, condition (b) can be replaced by:
\begin{itemize}
 \item [(b)'] for some $0<\gamma<1$ and $C>0$, and for all Carleson windows $Q(I)$,
\[
 \#\bigl\{\lambda\in Q(I) : 2^{-(l+1)}|I|<1-|\lambda| < 2^{-l}|I|\bigr\}\leq C 2^{\gamma l}\ , \qquad l\geq 0.
\]
\end{itemize}

In \cite{S}*{Corollary 2} it is mentioned that it can also be replaced by:
\begin{itemize}
 \item [(b)'']  there exist $0<\gamma<1$ and such that $\Lambda$ is $\gamma$-Carleson. 
\end{itemize}

In view of conditions (a) and (b)'' the following characterization of Poisson processes which are a.s. Bloch interpolating sequences follows from Theorems~\ref{thm:separation} and ~\ref{thm:Carleson}(b) (with $\gamma\in (2/3,1)$).

\begin{theorem}
 Let $\Lambda_\mu$ be the Poisson process associated to a positive, $\sigma$-finite, locally finite measure $\mu$. Then, 
  \[
   \P\bigl(\Lambda_\mu\ \textrm{is $\mathcal B$-interpolating}\bigr) =
   \begin{cases}
   1\quad &\textrm{if $\ \displaystyle\sum\limits_{n,k}\mu_{n,k}^{3}<\infty$} \\
   0\quad &\textrm{if $\ \displaystyle\sum\limits_{n,k}\mu_{n,k}^{3}=\infty$}.
  \end{cases}
   \]
\end{theorem}

\emph{Note.} In case $\sum_{n,k}\mu_{n,k}^3<\infty$ it is also possible to prove (b)' directly, with the same methods employed in the proof of Theorem~\ref{thm:Carleson}. It is enough to prove the estimate for dyadic arcs $I_{n,k}$, and for those
\[
 \#\bigl\{\lambda\in Q(I_{n,k}) : 2^{-(l+1)}|I_{n,k}|<1-|\lambda| < 2^{-l}|I_{n,k}|\bigr\}\simeq \sum_{j : T_{n+l,j}\subset Q_{n,k}} X_{n+l, j}.
\]
In the above, the left hand side corresponds essentially to the number of points in the layer $Q(I_{n,k})\cap A_{n+l}$.
Thus, with $m=n+l$, (b)' is equivalent to
\[
 \sup_{n,k}\sup_{m\geq n} 2^{-\gamma(m-n)} \sum_{j : T_{m,j}\subset Q_{n,k}} X_{m,j}<+\infty.
\]
Letting 
\[
Y_{n,k,m}=2^{-\gamma(m-n)} \sum_{j : T_{m,j}\subset Q_{n,k}} X_{m,j}\ ,\quad \E[Y_{n,k,m}]=2^{-\gamma(m-n)} \sum_{j : T_{m,j}\subset Q_{n,k}} \mu_{m,j}
\]
and proceeding as in the first part of the proof of Theorem~\ref{thm:Carleson}(a) we get (taking $A=3$):
\begin{align*}
 \sum_{n,k}\sum_{m\geq n} \P\bigl(Y_{n,k,m}\geq 3\bigr)&\lesssim\sum_{n,k}\sum_{m\geq n}
 \Bigl[2^{-\gamma(m-n)} \sum_{j : T_{m,j}\subset Q_{n,k}} \mu_{m,j}\Bigr]^3\\
 &\leq \sum_{n,k}\sum_{m\geq n} 2^{-3\gamma(m-n)} \sum_{j : T_{m,j}\subset Q_{n,k}} \mu_{m,j}^3\ 2^{2(m-n)}\\
 &=\sum_{m,j} \mu_{m,j}^3 \sum_{n\leq m} \sum_{k : Q_{n,k} \supseteq T_{m,j}} 2^{-(3\gamma-2)(m-n)}.
\end{align*}
For any $\gamma>2/3$ this sum is bounded by $\sum_{m,j} \mu_{m,j}^3$, so we can conclude with the Borel-Cantelli lemma.

\subsection{Interpolation in Dirichlet spaces}

Our last set of results concerns interpolation in the Dirichlet spaces, 
\[
 \mathcal D_\alpha=\bigl\{f\in H(\D) : \|f\|_{\mathcal D_\alpha}^2=|f(0)|^2+\int_{\mathbb D} |f'(z)|^2 (1-|z|^2)^{\alpha} dm(z)<\infty\bigr\},
\]
with $\alpha\in(0,1)$. The limiting case $\alpha=1$ can be identified with the Hardy space $H^2$. 

In these spaces, interpolating sequences are characterized by the separation and a Carleson type condition. This was initially considered by W.S. Cohn, see \cite{Coh}; we refer also to the general result \cite{AHMR}. While separation is a simple condition, that in our random setting is completely characterized by Theorem \ref{thm:separation}, the characterization of Carleson measures in these spaces is much more delicate. This was achieved by D. Stegenga using the so-called $\alpha$-capacity \cite{St}. In our setting it is however possible to use an easier sufficient one-box condition that can be found in K. Seip's book, see \cite{S}*{Theorem 4, p.38}, which we recall here for the reader's convenience.

\begin{theorem}[Seip]
A separated sequence $\Lambda$ in $\D$ is interpolating for $\mathcal D_{\alpha}$, $0<\alpha<1$ if there exist $0<\alpha'<\alpha$ 
such that $\Lambda$ is $\alpha'$-Carleson.
\end{theorem}

The reader should be alerted that in Seip's book the space $\mathcal D_{\alpha}$ is defined in a slightly different way, and that the above statement is adapted to our definition.

For these spaces Theorems \ref{thm:separation} and \ref{thm:Carleson} lead to less precise conclusions. Indeed, in view of Theorem~\ref{thm:Carleson}(c),(d) we cannot hope for complete characterizations if we do not impose additional conditions on the measure $\mu$% (see Section \ref{Examples} for a scale of measures for which a characterization can be given). %Still, we obtain:
.

\begin{theorem}\label{thm:Dirichlet}
 Let $\Lambda_\mu$ be the Poisson process associated to a positive, $\sigma$-finite, locally finite measure $\mu$.
 
 \begin{itemize}
  \item [(a)] If $1/2< \alpha<1$, then
\[
P\bigl(\Lambda_\mu \text{ is interpolating for $\mathcal D_\alpha$}\bigr)=
\begin{cases}
 1\quad \textrm{if $\ \displaystyle\sum\limits_{n,k} \mu_{n,k}^{2}<+\infty$}\\
 0\quad \textrm{if $\ \displaystyle\sum\limits_{n,k} \mu_{n,k}^{2}=+\infty$}.
\end{cases}
\]
  
\item [(b)] If $0\le \alpha<1/2$ and there exists $1<\gamma<\frac{1}{1-\alpha}$ such that $\sum\limits_{n,k} \mu_{n,k}^{\gamma}<+\infty$, then
\[
P\bigl(\Lambda_\mu \text{ is interpolating for $\mathcal D_\alpha$}\bigr)=1.
\]
 \end{itemize}
\end{theorem}

Clearly, the condition $\sum_{n,k}\mu_{n,k}^2<+\infty$ is also necessary in the case (b)  (if the sum diverges, then $\Lambda_{\mu}$ is almost surely not separated).

\begin{proof}%[Proof of Theorem \ref{thm:Dirichlet}]
(a) If $\sum_{n,k}\mu_{n,k}^2=+\infty$, then $\Lambda_{\mu}$ is almost surely not separated by Theorem \ref{thm:separation}, hence it is almost surely not interpolating.

If $\sum_{n,k}\mu_{n,k}^2<+\infty$, Theorem \ref{thm:separation}  shows again that the sequence $\Lambda_{\mu}$ is almost surely separated. By Seip's theorem, it remains to show that $\Lambda_{\mu}$ is almost surely $\alpha'$-Carleson for some $\alpha'<\alpha$.  Pick $1/2<\alpha'<\alpha<1$, so that $1/(1-\alpha')>2$. Choosing $\gamma\in (2,1/(1-\alpha'))$ we get 
\[
 \sum_{n,k}\mu_{n,k}^{\gamma}\lesssim \sum_{n,k}\mu_{n,k}^2<+\infty,
\]
and by Theorem \ref{thm:Carleson}(b) we conclude that $\Lambda_{\mu}$ is almost surely $\alpha'$-Carleson.

(b) If $\alpha<1/2$ then $1/(1-\alpha)<2$ and the value $\gamma$ given by the hypothesis satisfies $1<\gamma<2$. Therefore
\[
  \sum_{n,k}\mu_{n,k}^2\lesssim \sum_{n,k}\mu_{n,k}^{\gamma}<+\infty,
\]
which allows to deduce from Theorem \ref{thm:separation} that $\Lambda_{\mu}$ is almost surely separated.

Since the inequality $\gamma<1/(1-\alpha)$ is strict, we also have $\gamma<1/(1-\alpha')$ for some $\alpha'<\alpha$ sufficiently close to $\alpha$. Again, Theorem \ref{thm:Carleson}(b) shows that $\Lambda_{\mu}$ is almost surely $\alpha'$-Carleson, and Seip's theorem implies that $\Lambda_{\mu}$ is almost surely interpolating.
\end{proof} 

\subsection{Additional remarks and comments}

The above results show several applications of our Theorems \ref{thm:separation} and \ref{thm:Carleson}, but they also give rise to many challenging questions. Is it possible to get a necessary counterpart of Theorem~\ref{thm:Carleson}(b) under reasonable conditions on $\mu$ (more general than the class considered in Section \ref{Examples} below)? Is it possible to get precise statements when $\alpha=1/2$? Also, the case of the classical Dirichlet space seems to be largely unexplored for Poisson point processes, while the situation regarding interpolation, separation and zero-sets for the radial probabilistic model is completely known for all $\alpha\in [0,1]$ (see \cites{CHKW,Bog}). 

%%%%%%%%%%%%%%%%%%%%%%%%%%%%%%%%%%%%%%%%%%%%%%%%%%%%%%%%%%%%%%%%

\section{Examples and integral conditions for the measure $\mu$}
In the first part of this final section we illustrate the above results with three simple families of measures on $\D$. In the second part we briefly discuss alternative, non-discrete, formulations of the conditions given in the previous statements.

\subsection{Examples}\label{Examples}
\emph{1. Radial measures}. Let $dm$ denote the normalized Lebesgue measure and let $d\nu(z)=\frac{dm(z)}{(1-|z|^2)^2}$ be the invariant measure in $\D$. Define
\[
d\mu(a,b)(z)= \frac{dm(z)}{(1-|z|^2)^{a} \log^b\bigl(\frac e{1-|z|^2}\bigr)}= 
\frac{d\nu(z)}{(1-|z|^2)^{a-2} \log^b\bigl(\frac e{1-|z|^2}\bigr)},
\]
where either $a> 1$, $b\in\R$, or $a=1$ and $b\leq 1$ (so that $\mu(a,b)(\D)=+\infty$).

Observe that
\[
 \mu(a,b)_{n,k}\simeq \frac{2^{-n(2-a)}}{n^b}\qquad n\geq 1,\ k=0, \dots, 2^n-1,
\]
and therefore, for $\gamma>0$,
\begin{equation}\label{ex:radial}
 \sum_{n,k} \mu(a,b)_{n,k}^\gamma\simeq \sum_n 2^n \frac{2^{-n(2-a)\gamma}}{n^{b\gamma}}=
 \sum_n\frac{2^{-n[(2-a)\gamma-1]}}{n^{b\gamma}}.
\end{equation}

\begin{proposition} Consider the Poisson process $\Lambda_{a,b}$ associated to the masure $\mu(a,b)$, with either $a>1$ or $a=1$ and $b\leq 1$.
 \begin{itemize}
  \item [(a)] $\Lambda_{a,b}$ can a.s.  be expressed as a union of $M$ separated sequences if and only if either $a<2-\frac 1{M+1}$ and $b\in\R$, or $a=2-\frac 1{M+1}$ and $b>\frac 1{M+1}$.
  
\item [(b)]  In particular, $\Lambda_{a,b}$ is a.s. separated  if and only if either $a<3/2$ and $b\in\R$, or $a=3/2$ and $b> 1/2$.

\item [(c)] $\Lambda_{a,b}$ is a.s. a 1-Carleson sequence if and only if $a<2$, $b\in\R$.

\item [(d)] Let $\alpha\in (0,1)$. Then
 $\Lambda_{a,b}$ is a.s. an $\alpha$-Carleson sequence if and only if $a<1+\alpha$ or $a=1+\alpha$ and $b> 1$. 
 \end{itemize}

\end{proposition}

\begin{proof}
 (a) is immediate from Theorem~\ref{thm:separation} and \eqref{ex:radial} with $\gamma=M+1$. 
 
 (c) If $a\geq 2$ the series in \eqref{ex:radial} diverges for all $\gamma>1$, thus by Theorem~\ref{thm:Carleson}(a) $\Lambda_{a,b}$ is a.s. not  1-Carleson.
 
 On the other hand, if $a<2$ there exists $\gamma$ such that $(2-a)\gamma-1>0$ (i.e, such that $\gamma>\frac 1{2-a})$. For that $\gamma$ the series in \eqref{ex:radial} converges, and we can conclude again by Theorem~\ref{thm:Carleson}(a).
 
 (d) Suppose first that $a<1+\alpha$. As in the previous case, since $2-a>1-\alpha$ there exists $\gamma\in(\frac 1{2-a},\frac 1{1-\alpha})$. For this $\gamma$ the series in \eqref{ex:radial} converges and we can apply Theorem~\ref{thm:Carleson}(b). 

If $a>1+\alpha$ and $b\in\R$, then  $\Lambda_{\mu(a,b)}$ contains in the mean more points than $\Lambda_{\mu(1+\alpha,1)}$ for which we have shown in Theorem \ref{thm:Carleson}(c) that it is almost surely not $\alpha$-Carleson. 

It thus remains the case $a=1+\alpha$. Again, when $b=1$ --- and thus also when $b<1$ since then we have more points in the mean ---  the proof of Theorem~\ref{thm:Carleson}(c) shows that the corresponding sequence is almost surely not $\alpha$-Carleson.

Finally, suppose that $a=1+\alpha$ and $b>1$. Recall from \eqref{ynk-alpha} the notation
\[ %begin{equation}\label{ynk-alpha}
 \quad Y_{n,k}=2^{n\alpha}\sum_{m\geq n} 2^{-m\alpha} \sum_{j: T_{m,j}\subset Q_{n,k}} X_{m,j}.
\]
In the proof of Theorem \ref{thm:Carleson}(b) we have shown that
\[
 P(Y_{n,k}\ge A)\le B_{n,k}^A, 
\]
where
\[
 B_{n,k}=\sum_{m\geq n}  2^{-(m-n)\alpha} \sum_{j: T_{m,j}\subset Q_{n,k}} \mu_{m,j}.
\]
From the explicit form of $\mu_{m,j}$ we get
\[
 B_{n,k}\simeq \sum_{m \ge n}2^{-(m-n)\alpha}\times 2^{m-n}\times \frac{2^{-m(2-a)}}{m^b}
 =2^{n(\alpha-1)}\sum_{m\ge n}\frac{2^{-m(\alpha+1-a)}}{m^b}
%\simeq \frac{2^{-n(2-a)}}{n^b}.
\]
which converges exactly when $a<1+\alpha$ or when $a=1+\alpha$ and $b>1$ which is the case we are interested in here. In this situation, we get
\[
 B_{n,k}\simeq \frac{2^{-n(2-a)}}{n^b}
\]
Clearly, when $A\ge 1/(2-a)=1/(1-\alpha)$, then $\sum_{n,k}B_{n,k}^A$ converges, and the Borel-Cantelli lemma shows that $Y_{n,k}\ge A$ can happen for an at most finite number of Carleson windows $Q_{n,k}$. Hence $\Lambda_{\mu(a,b)}$ is a.s. $\alpha$-Carleson.
\end{proof}

%\begin{remarks}
%The example in Theorem \ref{thm:Carleson}(c) shows that when $a=1+\alpha$ and $b=1$, then $\Lambda_{\mu}$ is almost surely not $\alpha$-Carleson. It is clear that when we replace $1+\alpha$ by $a\ge 1+\alpha$ and $b=1$ by $b\leq 1$ we increase the number of points in the mean, which decreases the probability of being $\alpha$-Carleson. Thus the sequences $\Lambda_{a,b}$ for such cases are neither
%$\alpha$-Carleson. However, the result in the previous sections do not resolve the cases $a=1+\alpha$ and $b>1$. %At least the tools elaborated in this paper do not allow to conclude in this case.
%\end{remarks}
%\emph{Note.} (THIS SHOULD BE WRITTEN PROPERLY). The example in Theorem~\ref{thm:Carleson}(c). is of this family.

 \emph{ 2. Measures with a singularity on $\T$}. Define now
  \[
d\sigma(a,b)(z)= \frac{dm(z)}{|1-z|^{a} \log^b\bigl(\frac e{|1-z|}\bigr)},
\]
where either $a>2$, $b\in\R$, or $a=2$ and $b\leq 1$ (so that $\sigma(a,b)(\D)=+\infty$). 
Here
\begin{align*}
 \sigma (a,b)_{n,k}=\sigma (a,b)(T_{n,k})\simeq \frac{2^{-2n}}{[(k+1)2^{-n}]^a \log^b\bigl(\frac e{(k+1)2^{-n}}\bigr)},\quad n\in\N,\ k=0,\ldots, 2^{n-1}.
% \simeq \frac{2^{-n(2-a)}}{(k+1)^a\, n^b}.
\end{align*}
Hence for $\gamma>1$, %and since $a>2$,
\begin{equation}\label{ex:radial1}
 \sum_{n,k} \sigma (a,b)_{n,k}^\gamma\simeq
 \sum_{n} {2^{-n\gamma(2-a)}} \sum_{k=1}^{2^n}\frac 1{\Big(k^a\log^b\bigl(\frac e{k2^{-n}}\bigr)\Big)^{\gamma}}.
% \simeq\sum_{n} \frac{2^{-n\gamma(2-a)}}{ n^{\gamma b}},
\end{equation}
%and the associated Poisson process $\tilde\Lambda_{a,b}$ does not have good properties.

Let us examine the growth of the sum in $k$. For that, set
\[
 S_n(a,b,\gamma)=\sum_{k=1}^{2^n}\frac 1{k^{a\gamma}\log^{b\gamma}\bigl(\frac e{k2^{-n}}\bigr)}
 \simeq \int_1^{2^n}\frac{dx}{x^{a\gamma}\log^{b\gamma}\bigl(\frac e{x2^{-n}}\bigr)}
\]
The change of variable $t=\log\bigl(\frac e{x2^{-n}}\bigr)$ leads to
\[
S_n(a,b,\gamma)\simeq \int_{\log(2^ne)}^1 \left(\frac{e^t}{e2^n}\right)^{a\gamma-1}\frac{-dt}{t^{b\gamma}}
 =\frac{2^{-n(a\gamma-1)}}{e^{a\gamma -1}}\int_{1}^{\log(2^ne)}e^{t(a\gamma-1)}\frac{dt}{t^{b\gamma}}
\]
Our standing assumption being $a>2$ or $a=2$ and $b\le 1$, we only need to consider these two cases. In both cases, $e^{t(a\gamma-1)}/t^{b\gamma}\to +\infty$ when $t\to+\infty$, and the last integral behaves essentially as the value in the upper bound of the integration interval
\[
 \int_{1}^{\log(2^ne)}e^{t(a\gamma-1)}\frac{dt}{t^{b\gamma}}
\simeq \frac{2^{n(a\gamma-1)}}{n^{b\gamma}}.
\]
Hence
\[
 S_n(a,b,\gamma)\simeq \frac{1}{n^{b\gamma}},
\]
and
\begin{equation}\label{ex:radial3}
\sum_{n,k} \sigma (a,b)_{n,k}^\gamma\simeq \sum_n 2^{-n\gamma(2-a)}\times 
 \frac{1}{n^{\gamma b}}=\sum_n \frac{2^{-n\gamma(2-a)}}{n^{\gamma b}}.
\end{equation}
We are now in a position to prove the following result.

\begin{proposition}\label{Prop:Ex2}
Consider the Poisson process $\tilde \Lambda_{a,b}$ associated to the measure $\sigma(a,b)$, with either $a>2$ or $a=2$ and $b\leq 1$.
 \begin{itemize}
  \item [(a)] For $a>2$ the process $\tilde \Lambda_{a,b}$ is a.s.\ neither a finite union of separated sequences nor an $\alpha$-Carleson, for any $\alpha\in (0,1]$.
  
\item [(b)]  For $a=2$ the process $\tilde \Lambda_{2,b}$ is 
\begin{itemize}
 \item [(i)] the union of $M$ separated sequences if and only if $b>\frac 1{M+1}$,
 \item [(ii)]  $\alpha$-Carleson for $\alpha\in (0,1)$ if $b>1-\alpha$.
\end{itemize}
 \end{itemize}
\end{proposition}

\begin{proof}
 (a) is immediate from Theorems~\ref{thm:separation} and ~\ref{thm:Carleson}, since \eqref{ex:radial3} diverges for all $\gamma>0$. 
 
 (b) In this case the series \eqref{ex:radial3} is just $\sum_n 1/n^{b\gamma}$. 

The case (i) follows from Theorem~\ref{thm:separation} with $\gamma=M+1$.

For (ii), by the hypothesis  $1/b<1/(1-\alpha)$, there exists $1/b<\gamma<1/(1-\alpha)$, for which the series \eqref{ex:radial3} converges. We can conclude by Theorem~\ref{thm:Carleson}.
\end{proof}

\emph{ 3. Measures in a cone}. Given a point  $\zeta\in\T$, consider a Stolz region 
 \[
 \Gamma(\zeta)=\bigl\{z\in \D :\frac{|\zeta-z|}{1-|z|}<2\bigr\}.
 \]
We discuss the previous measures restricted to $\Gamma(\zeta)$. With no restriction of generality we can assume that $\zeta=1$. Let thus
 \[
d\tau(a,b)(z)= \chi_{\Gamma(1)}(z) d\mu(a,b)(z)=\chi_{\Gamma(1)}(z)\frac{dm(z)}{(1-|z|^2)^{a} \log^b\bigl(\frac e{1-|z|^2}\bigr)},
\]
where now either $a> 2$, $b\in\R$, or $a=2$ and $b\leq 1$ (so that $\nu(a,b)(\D)=+\infty$). Since in $\Gamma$ the measures $d\mu(a,b)$ and $d\sigma(a,b)$ behave similarly, we could replace $d\mu(a,b)$ by $d\sigma(a,b)$ in the definition of $d\tau(a,b)$.

Observe that $\nu(a,b)_{n,k}$ is non-zero only for a finite number $N$ of $k$ at each level $n$, and that for those $k$
\[
 \tau(a,b)_{n,k}\simeq \frac{2^{-n(2-a)}}{n^b}\qquad n\geq 1,\ k=0, \dots, N.
\]
Hence
\begin{equation}\label{ex:radial2}
 \sum_{n,k} \tau(a,b)_{n,k}^\gamma\simeq \sum_n \frac{2^{-n(2-a)\gamma}}{n^{b\gamma}},
\end{equation}
which is exactly the same estimate as in \eqref{ex:radial3} and thus immediately leads to the same result as Proposition \ref{Prop:Ex2}. This might look surprising since $\sigma(a,b)$ (and a fortiori $\mu(a,b)$) puts infinite mass outside $\Gamma(\zeta)$ (actually outside Stolz angles at $\zeta$ with arbitrary opening).

\begin{proposition} Consider the Poisson process $\hat \Lambda_{a,b}$ associated to the masure $\tau(a,b)$, with either $a>2$ or $a=2$ and $b\leq 1$.
 \begin{itemize}
  \item [(a)] For $a>2$ the process $\hat \Lambda_{a,b}$ is a.s. not a finite union of separated sequences
  separated or $\alpha$-Carleson for any $\alpha\in (0,1]$.
  
\item [(b)]  For $a=2$ the process $\hat \Lambda_{2,b}$ is 
\begin{enumerate}
 \item the union of $M$ separated sequences if and only if $b>\frac 1{M+1}$,
 \item $\alpha$-Carleson for $\alpha\in (0,1)$ if $b>1-\alpha$.
\end{enumerate}
 \end{itemize}
\end{proposition}

\subsection{Integral conditions on $\mu$}
Given a locally finite, $\sigma$-finite measure $\mu$, a natural question is whether the discretized conditions appearing in Theorems \ref{thm:separation} and \ref{thm:Carleson} can be reformulated in terms of integrals. Let us assume that $\mu$ is absolutely continuous with respect to the Lebesgue (or the invariant) measure on $\D$. In view of the aforementioned discrete conditions this is not really restrictive, since in case $\mu$ had a singular part we could just redistribute its mass continuously on each $T_{n,k}$. Assume thus that $d\mu=h\, d\nu$, where $d\nu=\frac{dm(z)}{(1-|z|^2)^2}$ is the invariant measure, $h\geq 0$ and $h\in L^1_{loc }(\D; \nu)$ (but $h\notin L^1(\D; \nu)$, so that $\mu(\D)=\infty$). 

As a result of Jensen's inequality applied to $\nu$ on $T_{n,k}$, we deduce the following general observation.

\begin{proposition}
Let $\mu=h\, d\nu$, where $h\geq 0$ and $h\in L^1_{loc }(\D)$. For every $\gamma>1$ there exists $C>0$ such that
%\begin{enumerate}
%\item 
%If $M_1=
\begin{equation}\label{intcond}
 \sum_{n,k}\mu_{n,k}^{\gamma}\leq C \int_{\D} h^{\gamma}(z)\, d\nu(z).
\end{equation}
%$, then $\sum_{n,k}\mu_{n,k}^{\alpha}$ converges and is 
%controlled by $M_1$.
%where the underlying constant is universal.
%\item If moreover $h$ is subharmonic in $\D$,
%=|g|$, where $g$ is holomorphic and 
%then
%\[
% \sum_{n,k}\mu_{n,k}^{\alpha}\simeq \int_{\D}h^{\alpha}
%(1-|z|^2)^{2(\alpha-1)}dm(z)<\infty,
%\]
%where the underlying constants are universal.
% and $M_2=\sum_{n,k}\mu_{n,k}^{\alpha}$ 
%converges, then $\int_{\D}h^{\alpha}
%(1-|z|^2)^{2(\alpha-1)}dm(z)$ converges and is controlled by $M_2$.
%\end{enumerate}
\end{proposition}

Of course without additional conditions on $h$ the conditions on the sum and on the integral cannot be equivalent, and there are standard construction methods to find measures for which the sum is convergent while the integral diverges. We do not go into details of that here. 

One is obviously more interested in situations where the sum and the integral conditions are equivalent. This is for instance the case when 
$h$ is radial with some regularity conditions (as in the first class of examples in the previous section) or when $h$ is
subharmonic  (as in the second class of examples).

Another mildly regular case in which an equivalent reformulation in terms of integrals is possible is when $\mu$ is doubling, meaning that there exists $C>0$ such that $\mu(2D)\leq C \mu(D)$ for all open disc $D\subset\D$. Here $2D$ denotes the disk with the same center as $D$ but with double radious (in the pseudohyperbolic metric). Fixing any $c\in (0,1)$ and defining $F_\mu(z)=\mu(D(z,c))$ one immediately sees that for any $\gamma>1$
\[
 \sum_{n,k}\mu_{n,k}^\gamma \simeq \int_{\D} F_\mu^\gamma (z)\, d\nu(z).
\]

\end{document}